\documentclass[a4paper]{amsart}
\usepackage{bbm}
\usepackage{graphicx}
\usepackage{amssymb}
\usepackage{amsmath}
\usepackage{amsthm}
\usepackage{amscd}
\usepackage[all,2cell]{xy}
\usepackage{bbm}
\usepackage{amsfonts}
\usepackage{lineno}
\usepackage{xypic}
\usepackage{color}
\allowdisplaybreaks[4]
\UseAllTwocells \SilentMatrices
\newtheorem{thm}{Theorem}[section]
\newtheorem{cor}[thm]{Corollary}
\newtheorem{lem}[thm]{Lemma}

\newtheorem{prop}[thm]{Proposition}
\newtheorem{prop-def}[thm]{Proposition-Definition}
\theoremstyle{definition}
\newtheorem{Def}[thm]{Definition}
\theoremstyle{remark}
\newtheorem{rmk}[thm]{\bf Remark}
\newtheorem{ex}[thm]{\bf Example}
\numberwithin{equation}{section}

\def\op{{\rm op}}
\def\Z{\mathbb{Z}}

\def\al{\alpha}

\def\xr{\xrightarrow}

\def\Hom{{\rm Hom}}
\def\C{\mathcal{C}}
\def\A{\mathcal{A}}
\def\B{\mathcal{B}}
\def\P{\mathcal{P}}

\def\Id{{\rm Id}}

\def\A{\mathcal{A}}
\def\Sum{\sum\limits}
\def\Ker{\mbox{\rm Ker}}
\def\Tot{\Xi}

\def\cone{{\rm Con}}

\def\Ker{{\rm Ker}\,}
\def\Coker{{\rm Coker}\,}
\def\E{\mathcal{E}}

\def\bu{{\bullet}}

\def\Proj{{\rm Proj}\mbox{-}}
\def\proj{{\rm proj}\mbox{-}}
\def\La{\Lambda}

\def\Sum{\mbox{\rm Sum-}}

\def\Th{\Theta}
\def\grb{\mathbb{Z}_+^b}
\def\gr{\mathbb{Z}_+}
\def\K{K}

\def\add{{\rm add}\mbox{-}}

\begin{document}
\title[A new Frobenius exact structure on the category of complexes]
{A new Frobenius exact structure on the category of complexes}

\author[Y.-F. Ben\quad Y.-H. Bao and X.-N. Du]
{Yan-Fu Ben\quad \quad Yan-Hong Bao$^*$ and Xian-Neng Du}

\thanks{$^*$ Corresponding author: Y.-H. Bao (Email: baoyh@ahu.edu.cn)}
\date{\today}

\maketitle

\dedicatory{}%
\commby{}%

\begin{abstract}
Let $(1)$ be an automorphism on an additive category $\B$, and let
$\eta\colon (1)\to \Id_{\B}$ be a natural transformation satisfying
$\eta_{X(1)}=\eta_X(1)$ for any object $X$ in $\B$. We construct a new Frobenius exact structure on the category of complexes in $\B$, which is associated to the natural transformation $\eta$. As a consequence, the category introduced in Definition 2.4 of [J. Rickard, Morita theory for derived categories, J. London Math. Soc. 39(1989), 436--456] has a Frobenius exact structure.
\end{abstract}

\section{Introduction}

Let $\B$ be an additive category and $C(\B)$ the category of
complexes in $\B$. It is well known that $C(\B)$ has a Frobenius exact structure, where the conflations are given by chainwise split short exact sequences of complexes. The corresponding stable category coincides with the homotopy category $K(\B)$ of complexes, which is a triangulated category.

Suppose that the category $\B$ has a certain
automorphism endofunctor $(1)$ and that $\eta\colon (1) \to \Id_{\B}$
is a natural transformation, which satisfies
$\eta_{X(1)}=\eta_{X}(1)$ for any object $X$ in $\B$. Associated to the natural transformation $\eta$, we construct a new Frobenius exact structure on $C(\B)$, whose conflations are given by certain chainwise split short exact sequences of complexes; see Theorem \ref{Frob cat (1)}. Consequently, the corresponding stable category $K_\eta(\B)$ is a triangulated category by \cite[Theorem I.2.8]{Ha1}. To describe the morphisms in $K_\eta(\B)$, we define a new homotopy relation on chain maps. In particular, the stable category $K_\eta(\B)$ is different from $K(\B)$, and thus is a new triangulated category.

This work is inspired by the construction of the functor $\Phi$ in \cite[Proposition 2.10]{Ric}, where the functor $\Phi$
induces a triangle functor in \cite[Proposition 2.11]{Ric} and then gives rise to the famous derived Morita theory.  More precisely, for an additive category $\mathcal{A}$, Rickard introduces an additive category $G(\A)$ in \cite[Definition 2.4]{Ric} and then defines $\Phi$ as the composition of an operation $(-)^{**}$ from a category of complexes to $G(\A)$
and a ``total" functor from $G(\A)$ to a homotopy category of complexes . However, the operation $(-)^{**}$ is not a functor.  Recall that a homotopy relation on morphisms in $G(\A)$ is introduced in \cite[Definition 2.8]{Ric}. Even up to this homotopy relation, the operation $(-)^{**}$ is still not a functor.

Theorem \ref{Frob cat (1)} gives rise to a new homotopy relation on morphisms in $G(\A)$, which makes the operation $(-)^{**}$ a triangle functor up to homotopy. Consequently, the triangle functor in \cite[Proposition 2.11]{Ric} can be realized as the composition of two triangle functors; see Proposition \ref{realization}.

Indeed, we observe that $G(\A)$ is isomorphic to $C(\B)$ for some additive category $\B$, which is endowed with a natural transformation $\eta\colon (1) \to \Id_{\B}$; see Proposition \ref{C(grA) and G(A)}.  Then by Theorem \ref{Frob cat (1)}, $G(\A)$ admits a Frobenius exact structure associated to $\eta$; moreover, we obtain a new homotopy relation on morphisms in $G(\A)$. Using this new homotopy relation, the operation $(-)^{**}$  induces a triangle functor.

This paper is organized as follows. In Section 2, we recall some basic facts on Frobeinus categories. Section 3 deals with
the construction of the new Frobenius exact structure on the category of complexes; see Theorem \ref{Frob cat (1)}.
In Section 4, we show that the category $G(\A)$ is isomorphic to a category of complexes in an additive category $\B$ endowed with a natural transformation $\eta\colon (1)\to \Id_{\B}$ and give the new homotopy relation on morphisms by Theorem \ref{Frob cat (1)}. Finally, we realize
the triangle functor in \cite[Proposition 2.11]{Ric} as the composition of two triangle functors in Theorem \ref{realization}.

Throughout this paper, all functors are additive functors.

\section{Preliminaries on Frobenius categories}

In this section we recall some basic facts on exact categories
and Frobenius categories; see \cite[Appendix A]{Ke3} for details.

Let $\mathcal{C}$ be an additive category. An \emph{exact pair} of morphisms
is a sequence $X\xr{i} Y \xr{p} Z$, which satisfies that $i=\Ker
p$ and $p=\Coker i$. Two exact pairs $(i, p)$ and $(i', p')$
are said to be \emph{isomorphic} provided that there exist
isomorphisms $f\colon X\to X'$, $g\colon Y\to Y'$ and $h\colon Z\to
Z'$ such that the following diagram commutes.
\[\begin{CD}
X @>i>> & Y @>p>> & Z\\
@VVfV & @VVgV &@VVhV\\
X' @>i'>> & Y' @>p'>> & Z'.
\end{CD}\]

Recall that an \emph{exact structure} on an additive category $\mathcal{C}$
is a chosen class $\E$ of exact pairs in $\mathcal{C}$, which is closed under isomorphisms and satisfies the
following axioms (Ex0), (Ex1), (Ex1)$^\op$, (Ex2) and (Ex2)$^\op$. An exact pair $(i, p)$ in the class $\E$ is called a \emph{conflation},
$i$ is called
an \emph{inflation} and $p$ is called a \emph{deflation}. The pair
 $(\mathcal{C}, \E)$ (or $\mathcal{C}$ for short) is called an \emph{exact category}.
 The axioms of exact categories are listed as follows:

 \begin{tabular}{ll}
(Ex0)&  the identity morphism of the zero object is a deflation;\\
(Ex1)& the composition of two deflations is a deflation;\\
(Ex1)$^\op$ & the composition of two inflations is an inflation;\\
(Ex2)& for a deflation $p\colon Y\to Z$ and a morphism $f\colon Z'\to Z$,
there exists a \\
& pullback diagram such
 that $p'$ is a deflation:
\end{tabular}
\[\xymatrix{
Y' \ar@{.>}[r]^{p'}\ar@{.>}[d]_{f'} & Z'\ar[d]^{f}\\
Y \ar[r]^p & Z.}\]
 \begin{tabular}{ll}
(Ex2)$^\op$ & for an inflation $i\colon X\to Y$ and a morphism $f\colon X\to X'$,
there exists a\\
& pushout diagram such
 that $i'$ is an inflation:
\end{tabular}
\[\xymatrix{
X \ar[r]^{i}\ar[d]_{f} & Y\ar@{.>}[d]^{f'}\\
X' \ar@{.>}[r]^{i'} & Y'.}\]

Let $P$ be an object in $\mathcal{C}$. If for any deflation $p \colon Y\to Z$ and any
morphism $f\colon P\to Z$, there exists a morphism $f'\colon
P\to Y$ such that $f=pf'$, then we call that $P$ is a
\emph{projective object} in the exact category $\mathcal{C}$. For any object $X$
in $\mathcal{C}$, if there exists a deflation $p \colon P\to X$ with $P$
projective, then we say that exact category $\mathcal{C}$ \emph{has enough
projective objects}. Dually, one can define \emph{injective objects}
and \emph{having enough injective objects} for $\mathcal{C}$. An exact category $\mathcal{C}$ is said to be \emph{Frobenius} provided that it has enough projective and enough injective objects, and the class of projective objects coincides with the class of injective objects. By a \emph{Frobenius exact structure} on an additive category $\C$, we mean a chosen class $\E$ such that
$(\C, \E)$ is a Frobenius category.

In what follows, we recall the stable category $\underline{\mathcal{C}}$ of a Frobenius category $\mathcal{C}$. The objects
of $\underline{\C}$ are the same as the ones in $\mathcal{C}$,  and the set of morphisms are defined as
\[\Hom_{\underline{\mathcal{C}}}(X, Y)=\Hom_{\mathcal{C}}(X, Y)/I(X, Y),\]
where $I(X, Y)=\{f\in \Hom_{\mathcal{C}}(X, Y)\mid f\,\,\mbox{factors through
some injective object}\}$ is the subgroup of $\Hom_{\mathcal{C}}(X, Y)$,
and the composition is defined naturally. For any object $X$,
we fix a conflation $X\xr{i_X} I(X) \xr{p_X} S(X)$, where $I(X)$ is
an injective object. For any morphism $f\colon X\to Y$, there exists a morphism $S(f)\colon S(X)\to S(Y)$ such that the following diagram commutes
\[\xymatrix{
X \ar[r]^{i_X}\ar[d]_{f} & I(X)\ar[r]^{p_X}\ar@{.>}[d] & S(X)\ar@{.>}[d]^{S(f)}\\
Y \ar[r]^{i_Y} & I(Y) \ar[r]^{p_Y} & S(Y).}\]
Here, the existence the middle morphism is deduced from the injectivity of $I(Y)$.
Note that $S(f)$ is not unique, but it is uniquely determined by
$f$ in the stable category $\underline{\mathcal{C}}$. This gives rise to
the \emph{suspension functor}  $S\colon \underline{\mathcal{C}} \to \underline{\mathcal{C}}$ ;
it is an auto-equivalence.

By \cite[Theorem I.2.8]{Ha1}, the stable category  $\underline{\C}$ of a Frobenius category $\C$ is a triangluated category, where
the translation functor is given by $S$. The triangles in $\underline{\C}$ are induced by conflations.

A classical example of Frobenius categories is the  category of complexes in an additive
category. Let $\B$ be an additive category, and let $\C(\B)$ be the category of complexes in $\B$.
The shift functor $[1]$ on $\C(\B)$ is defined as follows:
for any complex $X^\bu$ in $\C(\B)$,
the complex $X^\bu[1]$ is given by $(X^\bu[1])^n=X^{n+1}$
and $d_{X[1]}^n=-d_X^{n+1}$; for any chain map $f^\bu\colon X^\bu\to Y^\bu$, the morphism $f^\bu[1]$ satisfies
$(f^\bu[1])^n=f^{n+1}$. Then $[1]$ is an automorphism of $\C(\B)$.
Let $f^\bu\colon X^\bu \to Y^\bu$ be a morphism. The \emph{mapping cone} $\cone(f^\bu)$
of $f^\bu$ is a complex defined by
$(\cone(f^\bu))^n=X^{n+1}\oplus Y^n$, and the differential
$d_{\cone(f^\bu)}^n=\begin{pmatrix}
-d_{X}^{n+1} & 0\\
f^{n+1} & d_Y^n
\end{pmatrix}$.
The category $\C(\B)$ has a well-known exact structure $\E$ such that
$(\C(\B), \E)$ is a Frobenius category, where $\E$ consists of
all chainwise split short exact sequences of complexes.

For a chain map $f^\bu \colon X^\bu \to Y^\bu$,
the sequence
\[Y^\bu \xr{\left(0\atop 1\right)} \cone(f^\bu) \xr{(1\ 0)} X^\bu[1]\]
is called a \emph{standard exact pair} associated to $f^\bu$.
In fact, each exact pair in $\E$ is
isomorphic to some standard exact pair. Suppose that $X^\bu\xr{i^\bu} Y^\bu\xr{p^\bu} Z^\bu$
is a conflation. Up to isomorphism, we can assume that
$Y^n=Z^n\oplus X^n$, and $i^n=\left(0\atop 1\right)$,
$p^n=(1\ 0)$. It follows that
$d_Y^n=\begin{pmatrix}
d_Z^n & 0\\
h^{n+1} & d_X^n
\end{pmatrix}$,
and $h^\bu\colon Z^\bu[-1] \to X^\bu$ is a chain map.
Then the exact pair $(i^\bu, p^\bu)$ is isomorphic to
the standard exact pair associated to $h^\bu$, and the class of
the chain map $h^\bu$ is unique up to homotopy and the homotopy class of $h^\bu$ is called a \emph{homotopy invariant} of $(i^\bu, p^\bu)$; see \cite[Definition 4.7]{Ive}.

\section{A new Frobenius exact structure on the category of complexes}

In this section, we construct a new Frobenius exact structure on the category of complexes,
and then obtain a new triangulated category.

Let us fix the following notation.
Let $\B$ be an additive category and $\C(\B)$ the category of
complexes in $\B$. The category $\C(\B)$ has a well-known exact structure $\E$ such that
$(\C(\B), \E)$ is a Frobenius category, where $\E$ consists of
all chainwise split short exact sequences of complexes.

 Suppose that $(1)$ is an automorphism endofunctor on $\B$ and
$\eta\colon (1) \to \Id_{\B}$ is
 a natural transformation satisfying $\eta_{X(1)}=\eta_{X}(1)$ for any object $X$ in $\B$,
 where $\Id_{\B}$ is the identity functor on $\B$.
 The inverse functor of $(1)$ is denoted by $(-1)$.

The automorphism $(1)$ and the natural transformation $\eta$
induce the ones on $\C(\B)$, which
are still denoted by $(1)$ and $\eta$.
Here are some elementary properties of the automorphism $(1)\colon C(\B)\to C(\B)$ and the natural transformation $\eta\colon (1) \to \Id_{C(\B)}$.

\begin{lem}\label{prop of eta}
For any objects $X^\bu$, $Y^\bu$ and any morphism $f^\bu\colon X^\bu\to Y^\bu$ in $\C(\B)$, we have

\item[(i)] $(1)[1]=[1](1)$;

\item[(ii)] $ \eta_{X^\bu(1)}=\eta_{X^\bu}(1)$, $ \eta_{X^\bu[1]}=\eta_{X^\bu}[1]$;

\item[(iii)]  $\cone(f^\bu(1))=\cone(f^\bu)(1)$;

\item[(iv)] $\eta_{\cone(f^\bu)}=
\begin{pmatrix}
\eta_{X^\bu[1]}& 0 \\
    0& \eta_{Y^\bu}
\end{pmatrix}$.\\

\end{lem}

\begin{proof} We give the proof only for (iv); the others are straightforward.

Recall that for any morphism $f^\bu \colon X^\bu \to Y^\bu$,
$(\cone(f^\bu))^{n}=X^{n+1}\oplus Y^n$, and
$d_{\cone(f^\bu)}^n=\begin{pmatrix} -d_X^{n+1} & 0 \\ f^{n+1} & d_Y^n
\end{pmatrix}$. By definition, $\eta_{\cone(f^\bu)}\colon \cone(f^\bu)(1)\to \cone(f^\bu)$ is a chain map such that
each component
$(\eta_{\cone(f^\bu)})^n\colon X^{n+1}(1)\oplus Y^n(1)\to X^{n+1}\oplus Y^n$ is given by
 $\begin{pmatrix}
\eta_{X^{n+1}} & 0\\
0 & \eta_{Y^n}
\end{pmatrix}$. This proves (iv).
\end{proof}

\begin{Def}\label{1-ex pair}
(1) Suppose that $f^\bu \colon X^\bu \to
Y^\bu$ is a chain map. If $f^\bu$ factors through
$\eta_{Y^\bu}\colon Y^\bu(1) \to Y^\bu$, then
\[Y^\bu \xr{\left(0\atop 1\right)} \cone(f^\bu) \xr{(1\ 0)} X[1]^\bu\]
is called a \emph{standard $\eta$-conflation}.

(2) We denote by $\E_\eta$ the subclass  of $\E$ consisting of all exact pairs that are isomorphic to some standard $\eta$-conflation.
An exact pair $(i^\bu, p^\bu)$ in $\E_\eta$ is called an \emph{$\eta$-conflation}, where $i^\bu$ is called an
\emph{$\eta$-inflation}, and $p^\bu$ is called an
\emph{$\eta$-deflation}.
\end{Def}

\begin{rmk}
The chain map $f^\bu \colon X^\bu \to Y^\bu$ factors
through $\eta_{Y^\bu}$ if and only if $f^\bu$ factors through
$\eta_{X^\bu(-1)}$ since $\eta\colon (1)\to \Id_{C(\B)}$ is a natural transformation.
\end{rmk}

\begin{lem}\label{exact structures}
The complex category $\C(\B)$ together with the class $\E_\eta$ is
an exact category.
\end{lem}

\begin{proof} The axiom (Ex0) is trivial. To show (Ex1), we assume
that $Y^\bu \xr{(1, 0)} Z^\bu$ and $W^\bu\xr{(1, 0)} Y^\bu$ are
$\eta$-deflations, where $Y^\bu=\cone(f^\bu)$ with the homotopy invariant $f^\bu\colon
Z^\bu[-1] \to X^\bu$, and $W^\bu=\cone(g^\bu)$ with the homotopy invariant  $g^\bu=(g_1^\bu,
g_2^\bu)\colon Y^\bu[-1] \to V^\bu$.
By simple calculation, we have that $g^\bu_2\colon  X^\bu[-1]
\to V^\bu$ and $\left({f^\bu}\atop{g^\bu_1}\right)\colon
Z^\bu[-1] \to \cone(g^\bu_2)$ are chain maps, and
$W=\cone(\left({f^\bu}\atop{g^\bu_1}\right))$.

By definition, we assume that $f^\bu=\eta_{X^\bu}\alpha^\bu$ and $g^\bu=\eta_{V^\bu}\beta^\bu$,
where $\alpha^\bu: Z^\bu[-1]\to X^\bu(1)$ and
$\beta^\bu=(\beta_1^\bu, \beta_2^\bu): Y^\bu[-1]\to V^\bu(1)$ are chain maps. Observe that
$\left({\alpha^\bu}\atop{\beta^\bu_1}\right)\colon Z^\bu[-1] \to
\cone(g^\bu_2)(1)$  is a chain map.
By Lemma \ref{prop of eta} (iv), we have
\[\begin{pmatrix} f^\bu \\ g^\bu_1 \end{pmatrix}
=\begin{pmatrix} \eta_{X^\bu} &0 \\ 0 & \eta_{V^\bu} \end{pmatrix}
\begin{pmatrix} \alpha^\bu \\ \beta^\bu_1 \end{pmatrix}
=\eta_{\cone(g_2^\bu)}\begin{pmatrix} \alpha^\bu \\ \beta^\bu_1
\end{pmatrix}.\]
It follows that the chain map
$\left({f^\bu}\atop{g^\bu_1}\right)\colon Z^\bu[-1] \to
\cone(g^\bu_2)$ factors through $\eta_{\cone(g^\bu_2)}$. Therefore, the composition $W^\bu\to Y^\bu $ and $Y^\bu \to Z^\bu$ is an $\eta$-deflation and
then (Ex1) holds.

To show (Ex2), for any chain map $h^\bu\colon
Z'^\bu\to Z^\bu$, we have the following pullback diagram
\[\xymatrix{
\cone(f^\bu h^\bu[-1]) \ar@{.>}[rr]^(0.6){(1,0)} \ar@{.>}[d]_{
{\footnotesize\begin{pmatrix}
h^\bu  & 0\\
0 & \Id_{X^\bu}
\end{pmatrix}}} && Z'^\bu\ar[d]^{ h^\bu}\\
Y^\bu \ar[rr]_{(1,0)} && Z^\bu ,}\] where $Y^\bu=\cone(f^\bu)$ with the homotopy invariant
$f^\bu\colon Z^\bu[-1] \to X^\bu$. Since $f^\bu h^\bu[-1]=\eta_{X^\bu}
\alpha^\bu h^\bu[-1]$ factors through $\eta_{X^\bu}$, we immediately have that (Ex2) hold.

Dually, (Ex1)$^\op$ and (Ex2)$^\op$ hold since the chain map $f^\bu \colon X^\bu \to Y^\bu$ factors through
$\eta_{Y^\bu}$ if and only if $f^\bu$ factors through
$\eta_{X^\bu(-1)}$. This finishes the proof.
\end{proof}

\begin{lem}\label{proj and inj obj} Let $V^\bu$ be an object
in $C(\B)$. Then the mapping cone $\cone(\eta_{V^\bu})$ is both
projective and injective in $(\C(\B), \E_\eta)$.
\end{lem}

\begin{proof} By definition, we have
\begin{center}
$(\cone(\eta_{V^\bu}))^{n}=V^{n+1}(1)\oplus V^{n}$, and
$d_{\cone(\eta_{V^\bu})}^n=\begin{pmatrix}
-d_{V}^{n+1}(1) & 0\\
\eta_{V^{n+1}} & d_{V}^{n}
\end{pmatrix}$.
\end{center}
Let $Y^\bu \xr{(1\ 0)} Z^\bu$
be any $\eta$-deflation in $\E_\eta$, where $Y^\bu=\cone(f^\bu)$ with the homotopy invariant
$f^\bu\colon Z^\bu[-1] \to X^\bu$. By definition, we assume that $f^\bu=\eta_{X^\bu}\al^\bu$,
 where $\al^\bu\colon  Z^\bu[-1]\to X^\bu(1)$ is a chain map.

To show that $\cone(\eta_{V^\bu})$ is a projective object,
it suffice to prove that $Y^\bu \xr{(1\ 0)} Z^\bu$ factors through any chain map $(g_1^\bu,g_2^\bu)\colon \cone(\eta_{V^\bu}) \to Z^\bu$.

Indeed, we have the following commutative diagram
\[\xymatrix{
 & \cone(\eta_{V^\bu})\ar[d]^{(g_1^\bu, g_2^\bu)} \ar@{.>}[dl]_(.7){
 {\footnotesize \begin{pmatrix}
 g_1^\bu & g_2^\bu\\
 0   & \alpha^\bu(-1)g_1^\bu[-1](-1)
 \end{pmatrix}}
 }\\
 Y^\bu \ar[r]^{(1,0)} & Z^\bu,}\]
and it remains to check that$\begin{pmatrix}
 g_1^\bu & g_2^\bu\\
 0   & \alpha^\bu(-1)g_1^\bu[-1](-1)
 \end{pmatrix}:\cone(\eta_{V^\bu}) \to Y^\bu$ is a chain map.
 Since $(g^\bu_1, g^\bu_2)$ is a chain map,  we have
\begin{align}
g_2^{n+1}\eta_{V^{n+1}}&=g_1^{n+1}d_V^{n+1}(1)+d_Z^ng_1^n, \label{comm-1}\\
g_2^{n+1}d_V^n&=d_Z^ng_2^n \label{comm-2}
\end{align}
for each $n\in \Z$.
Notice that $g_2^\bu$ is a chain map, but $g_1^\bu$ may be not a chain
map. By the following commutative diagram
\[\xymatrix{
\cone(\eta_{V^\bu}) \ar[rr]^{(g_1^\bu,g_2^\bu)}\ar[d]_{\footnotesize
\begin{pmatrix}
\eta_{V^\bu}& 0\\
0 & \eta_{V^\bu(-1)}
\end{pmatrix}} && Z^\bu\ar[d]^{\eta_{Z^\bu(-1)}}\\
\cone(\eta_{V^\bu})(-1) \ar[rr]_{(g_1^\bu(-1),g_2^\bu(-1))} &&
Z^\bu(-1),}\]
we have $g_1^n(-1)\eta_{V^{n+1}}=\eta_{Z^n(-1)}g_1^n$.
Therefore,
\begin{align}\label{comm-3}
\alpha^{n+1}(-1)g_1^n(-1)\eta_{V^{n+1}}
=\alpha^{n+1}(-1)\eta_{Z^n(-1)}g_1^n=\eta_{X^{n+1}}\alpha^{n+1}g_1^n=f^{n+1}g_1^n,
\end{align}
where the second equality is deduced from the naturality of $\eta$
since $\alpha^\bu\colon Z^\bu[-1]\to X^\bu(1)$ is a chain map. On the other hand,
we have
\begin{align}\label{comm-4}
\begin{split}
& d_X^n\alpha^n(-1)g_1^{n-1}(-1)+f^{n+1}g_2^n\\
=& \alpha^{n+1}(-1)(-d_Z^{n-1}(-1))g_1^{n-1}(-1)+\eta_{X^{n+1}}\alpha^{n+1}g_2^n\\
=& \alpha^{n+1}(-1)(-d_Z^{n-1}(-1))g_1^{n-1}(-1)+\alpha^{n+1}(-1)\eta_{Z^n(-1)}g_2^n\\
=& \alpha^{n+1}(-1)(-d_Z^{n-1}(-1))g_1^{n-1}(-1)+\alpha^{n+1}(-1)g_2^n(-1)\eta_{V^n(-1)}\\
=& \alpha^{n+1}(-1)\left(-d_Z^{n-1}(-1)g_1^{n-1}(-1)+g_2^n(-1)\eta_{V^n(-1)}\right)\\
=& \alpha^{n+1}(-1)g_1^n(-1)d_V^n,
\end{split}
\end{align}
where the second and the third
equalities are deduced from the naturality of $\eta$, and the fifth equality is a consequence of the equality
\eqref{comm-1}.
From the equalities \eqref{comm-2}, \eqref{comm-3} and
\eqref{comm-4}, it follows that $\begin{pmatrix}
 g_1^\bu & g_2^\bu\\
 0   & \alpha^\bu(-1)g_1^\bu[-1](-1)
 \end{pmatrix}$ is a chain map, and then
$\cone(\eta_{V^\bu})$ is a projective object
in $(\C(\B), \E_\eta)$.

Let $X^\bu\xr{\left(0\atop 1\right)} Y^\bu$ be any $\eta$-inflation,
where $Y^\bu=\cone(f^\bu)$ with $f^\bu: Z^\bu[-1]\to X^\bu$ factors
through $\beta^\bu: Z^\bu[-1]\to X^\bu(1)$. Dually, for any chain
map $\left({g_1^\bu}\atop{g_2^\bu}\right): X^\bu \to
\cone(\eta_{V^\bu})$, we have the following commutative diagram of
complexes
\[\xymatrix{
X^\bu \ar[rr]^{\left(0\atop
1\right)}\ar[d]_{\left({g_1^\bu}\atop{g_2^\bu}\right)} && Y^\bu
\ar@{.>}[dll]^(.3){{\footnotesize
\begin{pmatrix}
 g_2^\bu[1](1)\beta[1]  & g_1^\bu\\
  0 & g_2^\bu
\end{pmatrix}}}\\
\cone(\eta_{V^\bu}).}\]
Consequently, $\cone(\eta_{V^\bu})$ is also
an injective object.
\end{proof}

\begin{lem}\label{lem3.4}
Let $X^\bu \xr{i^\bu} Y^\bu \xr{p^\bu} Z^\bu$ be an exact pair in $\E$. Then the following statement are equivalent:

\noindent \emph{(1)} $(i^\bu, p^\bu)$ is an $\eta$-conflation;

\noindent \emph{(2)}  there exists a homotopy invariant $h^\bu \colon Z^\bullet[-1] \rightarrow X^\bu$  of $(i^\bu, p^\bu)$ which factors through $\eta_{X^\bu}$;

\noindent \emph{(3)} for any object $V^\bu$,
\[\Hom(\cone(\eta_{V^\bu}), Y^\bu)\xr{i^\ast} \Hom(\cone(\eta_{V^\bu}), Y^\bu)\xr{p^\ast} \Hom(\cone(\eta_{V^\bu}), Z^\bu)\]
 is short exact;

\noindent \emph{(4)} for any object $V^\bu$,
\[\Hom( Y^\bu, \cone(\eta_{V^\bu}))\xr{p_\ast} \Hom( Y^\bu, \cone(\eta_{V^\bu}))\xr{i_\ast} \Hom( X^\bu, \cone(\eta_{V^\bu}))\]
  is short exact.
\end{lem}
\begin{proof}
The equivalence of (1) and  (2) follows immediately from the definition of $\eta$-conflations.
By the projectivity (\emph{resp.} injectivity) of $\cone(\eta_{V^\bu})$
in the Frobenius category $(C(\B), E_\eta)$ for any object $V^\bu$, we know that (2) implies (3)
(\emph{resp.} (4)). Suppose that $Y^n=Z^n\oplus X^n$ and the differential
 $d_Y^n=\begin{pmatrix} d_Z^n & 0 \\ h^{n+1} & d_X^n\end{pmatrix}$ with the homotopy invariant $h^\bu\colon Z^\bu[-1] \to X^\bu$,
and $i^n=\left(0\atop 1\right)$, $p^n=(1\ 0)$ for each $n\in \Z$.

``(3) $\Longrightarrow$ (2)".  Take $V^\bu=Z^\bu(-1)[-1]$, and we have that
\[\Hom(\cone(\eta_{Z^\bu[-1](-1)}), Y^\bu)\xr{p^\ast} \Hom(\cone(\eta_{Z^\bu[-1](-1)}), Z^\bu)\]
 is an epimorphism of abelian groups. It follows that
 there exists a chain map $\footnotesize{\begin{pmatrix} 1 & 0 \\ \al^\bu & \beta^\bu\end{pmatrix}}
\colon \cone(\eta_{Z^\bu[-1](-1)})\to Y^\bu$ such that the following diagram commutes
\[\xymatrix{
& \cone(\eta_{Z^\bu[-1](-1)}) \ar@{.>}[dl]_(.7){\footnotesize{\begin{pmatrix} \Id_{Z^\bu} & 0 \\ \al^\bu & \beta^\bu\end{pmatrix}}}
\ar[d]^{(1, 0)}\\
Y^\bu \ar[r]^{(1, 0)} & Z^\bu}\]
Since $\footnotesize{\begin{pmatrix} 1 & 0 \\ \al^\bu & \beta^\bu\end{pmatrix}}$
is a chain map, we have \begin{align*}
h^{n}+d_X^{n-1}\al^{n-1}  =\al^{n}d_Z^{n-1}+\beta^{n}\eta_{Z^{n-1}(-1)}
\end{align*}
Set ${\widetilde h}^{n}=h^{n}+d_X^{n-1}\al^{n-1}+\al^{n}(-d_Z^{n-1})$, which is homotopic to $h^\bu \colon Z^\bu[-1] \to X^\bu$. Then we have $Y\bu=\cone(\widetilde{h}^\bu)$ and
${\widetilde h}^\bu =\beta^\bu\eta_{Z^\bu[-1](-1)}$ factors through $\eta_{Z^\bu[-1](-1)}$.
It follows from the naturality of $\eta$ that $\widetilde{h}^\bu$ factor through $\eta_{X^\bu}$.

``(4) $\Longrightarrow$ (2)" can be proved dually and we omit the proof.
\end{proof}

\begin{thm} \label{Frob cat (1)}
The exact category $(\C(\B), \E_{\eta})$ is a Frobenius category.
\end{thm}
\begin{proof}
Let $X^\bu$ be any object in $C(\B)$.
Since
$X^\bu\xr{\left(0\atop 1\right)} \cone(\eta_{X^\bu})$
 is an $\eta$-inflation, and
$\cone(\eta_{X^\bu[-1](-1)}) \xr{(1\ 0)} X^\bu$
  is an $\eta$-deflation, it follows from Lemma \ref{proj and inj obj} that
$(\C(\B), \E_{\eta})$ has enough projective and injective objects.

On the other hand, an object $X^\bu$ is injective if and only if
$X^\bu$ is a direct summand of $\cone(\eta_{X^\bu})$. By Lemma \ref{proj and inj obj},
$\cone(\eta_{X^\bu})$ is a projective object. It follows that $X^\bu$ is projective.
Dually, a projective object is also an injective object.
\end{proof}

We denote $K_{\eta}(\B)$ by the stable category of the Frobenius category
$(\C(\B), \E_{\eta})$. Recall that $K_\eta(\B)$ has the same objects as $\C(\B)$ and a
morphism in $K_{\eta}(\B)$ is the equivalence class $\overline{f^\bu}$ of
$f^\bu\colon X^\bu \to Y^\bu$ of $\C(\B)$ modulo the subgroup of
morphisms factoring through some injective object in $(\C(\B),
\E_{\eta})$.

By \cite[Theorem I.2.8]{Ha1}, we know that $K_{\eta}(\B)$ is a
triangulated category and $(1)[1]$ is the corresponding translation functor. Each
chain map $f^\bu \colon X^\bu \to Y^\bu$ induces a standard triangle
\[ X^\bu \xr{\overline{f^\bu}} Y^\bu
\xr{{\left(0\atop 1\right)}} \cone(f^\bu\eta_{X^\bu})
\xr{{(1\ 0)}} X^\bu(1)[1].\]

Next we give the homotopy relation on morphisms in $C(\B)$ given by the stable category of $(C(\B), \E_\eta)$.

\begin{Def}\label{def of homotopy}
Let $X^\bu$ and $Y^\bu$ be complexes in $\B$. The chain maps $f^\bu,
g^\bu\colon X^\bu \to Y^\bu$ are said to be \emph{$\eta$-homotopic},
denoted by $f^\bu \sim_{\eta} g^\bu$, provided that there is a
family of morphisms $\{s^n\colon X^{n}(1)\to Y^{n-1}\}_{n\in \Z}$ in
$\B$ such that
\[(f^n-g^n)\eta_{X^{n}}=s^{n+1}d_X^n(1)+d_Y^{n-1}s^{n}\]
for each $n\in \Z$.
A chain map $f^\bu\colon X^\bu \to Y^\bu$ is said to be
\emph{$\eta$-null-homotopic} if $f^\bu \sim_{\eta} 0$, where $0$ is
the zero chain map.
\end{Def}

Let $f^\bu, g^\bu\colon X^\bu \to Y^\bu$ be morphisms in $C(\B)$,
then $f^\bu \sim_{\eta} g^\bu$ if and only if $f^\bu\eta_{X^\bullet}\sim g^\bu\eta_{X^\bullet}$,
where ``$\sim$" is the homotopy relation on morphisms given by the stable category $(\C(\B), \E)$.

\begin{prop}\label{prop of homotopy}
A chain map $f^\bu\colon X^\bu \to Y^\bu$ is $\eta$-null-homotopic
if and only if $\overline{f^\bu}=0$ in the stable category
$K_{\eta}(\B)$.
\end{prop}
\begin{proof}
By definition, $\overline{f^\bu}=0$ in the stable category
$K_{\eta}(\B)$ if and only if $f^\bu$ factors through
$X^\bu\xr{\left(0\atop 1\right)} \cone(\eta_{X^\bu})$,
or equivalently, there exists a map $s^\bu\colon X^\bu[1](1) \to Y^\bu$ such that
$(s^\bu, f^\bu):\cone(\eta_{X^\bu}) \to Y^\bu$ is a chain map. Explicitly,
there exists $s^n\colon X^{n}(1)\to Y^{n-1}$ such that
\[f^{n}\eta_{X^{n-1}}=s^{n+1}d_X^{n}(1)+d_Y^{n-1}s^{n}\]
for any $n\in \Z$.
\end{proof}

\begin{ex}
(1) If $\eta$ is a natural isomorphism, then $\E_\eta$ consists
of all chainwise split short exact sequences of complexes, that is, $\E_\eta =\E$.

(2) If $\eta$ is a zero natural transformation, then $\E_\eta$
consists of all split short exact sequences of complexes and the corresponding stable category
is trivial.

\end{ex}

\begin{rmk}  Let $(\C, \E)$ be an exact category,
with enough project objects and enough injective objects.
Suppose that $\mathcal{P}'$ and $\mathcal{I}'$ are two full subcategories of $\C$. In \cite[Theorem 5.1]{Ch} for detail.
 Chen proved that the pair $(\C, \E')$ is a Frobenius category, where $\E'$ is the class of
 left $\P'$-acyclic conflations provide that $\P'$ and $\mathcal{I}'$ satisfy some certain conditions.

Take $\C=C(\B)$ together with the Frobenius exact structure $\E$ given by chainwise split short exact sequences of complexes,
and  $\mathcal{P}'=\mathcal{I}'$ is the smallest full additive subcategory of $C(\B)$ which contains $\cone(\eta_{X^\bu})$ for all $X^\bu\in C(\B)$ and is closed under direct summands and under isomorphisms.
From Lemma \ref{lem3.4}, it immediately follows that the Frobenius category $(C(\B),
\E')$ coincides with $(C(\B), \E_\eta)$.
\end{rmk}

\begin{rmk} Let $\mathcal{T}$ be a triangulated category, and let $\mathcal{X}$ be a full subcategory of $\mathcal{T}$, closed under isomorphisms, direct sums, and
 direct summands, which is both preenveloping
 and precovering. Denoted by $\mathcal{T}_{\mathcal{X}}$ the stable category of $\mathcal{T}$ associated to $\mathcal{X}$.
 In \cite{Jphi}, J$\phi$rgensen proved that $\mathcal{T}_{\mathcal{X}}$ is a pretriangulated category.
 Take $\mathcal{T}=K(\B)$ and $\mathcal{X}$ is the smallest full additive subcategory of $K(\B)$ which contain $\cone(\eta_{X^\bu})$ for all $X^\bu$ in $K(\B)$ and is closed under taking direct summands. One can check that $K(\B)_{\mathcal{X}}$ is exactly a triangulated
 category and coincides with the triangulate category $K_\eta(\B)$.
\end{rmk}

\begin{rmk} For any $m\geq 1$, we consider the composition of $m$ copies of the functor $(1)$ , denoted by $(m)$.
Clearly,
$\eta^{m}_X=\eta_X\eta_{X(1)}\cdots \eta_{X(m-1)}$ gives a natural
transformation from $(m)$ to $\Id_{\B}$. Repeated application of Theorem \ref{Frob cat (1)}
enable us to obtain the Frobenius exact structure
$\E_{\eta^m}$ on $\C(\B)$. Denote by $\P_{\eta^m}$ the full subcategory of
projective objects in the Frobenius category $(\C(\B),\E_{\eta^m})$.
Then we have the following observations:
\begin{center}
$\E \supset \E_{\eta} \supset \cdots \supset \E_{\eta^m} \supset \cdots$,
and
$\P \subset \P_{\eta} \subset \cdots \subset \P_{\eta^m} \subset \cdots$.
\end{center}
where $\P$ is the full subcategory of projective-injective objects in the
Frobenius category $(\C(\B),\E)$.

\end{rmk}

\section{An application}

In this section, we introduce a new category $\gr\A$ for any additive category $\A$ endowed with
an automorphism endofunctor $(1)$ and a natural transformation
$\eta\colon (1) \to \Id$ satisfying $\eta_{X(1)}=\eta_X(1)$ for any object $X$ in $\gr\A$.
The category $C(\gr\A)$ of complexes in $\gr\A$ is isomorphic to the category $G(\A)$ introduced in
\cite[Definition 2.4]{Ric}. By Theorem \ref{Frob cat (1)}, the categories $C(\gr\A)$ and $G(\A)$
admit a new Frobenius exact structure and the corresponding stable categories are triangulated categories.
Moreover, using the $\eta$-homotopy relation in Proposition \ref{1-null-homotopic}, we will realize the functor in \cite[Proposition 2.11]{Ric}
as the composition of two triangle functors.

\subsection{The category $\gr\A$}
\begin{Def}\label{def of B}
 Let $\A$ be an additive category. The category $\gr\A$ is defined as follows:

(1) its objects are $\Z$-graded objects $X^\bu$ with $X^{j}$ in $\A$
for each $j\in \Z$,

(2) a morphism from $X^\bu$ to $Y^\bu$ is a collection
 $f=\{f_0, f_1, \cdots, f_n,\cdots \}$ of homogeneous 
 maps from $X^\bu$ to $Y^\bu$,
 where $f_n$ has degree $n$ with $f^{j}_n\colon X^j\to Y^{j+n}$
 in $\A$ for all $n\geq 0, j\in\Z$,

(3) the composition of morphisms $ f\colon X^\bu \rightarrow Y^\bu$
and $g \colon Y^\bu \to Z^\bu$ is defined by
\[(g f)_{n}^j=g_{n}^j f_{0}^j+g_{n-1}^{j+1} f_{1}^{ j}+\cdots+g_{0}^{j+n}f_{n}^{j}\]
for $n\ge 0, j\in \Z$. The identity map is $\{\Id,0,0,\cdots\}$.
\end{Def}

We consider the functor $(1):\gr\A \to \gr\A$ given by
$(X^\bu{(1)})^{j}=X^{j+1}$, and $(f^\bu{(1)})_{n}^{j}=f_{n}^{j+1}$
for any object $X^\bu$  and any morphism $f \colon X^\bu \to Y^\bu$ in $\gr\A$.
Clearly, the functor $(1)$ is an automorphism with the inverse given by
$(X^\bu{(-1)})^{j}=X^{j-1}$, and $(f{(-1)})_{n}^{j}=f_{n}^{j-1}$
for any object $X^\bu$  and any morphism $f \colon X^\bu \to Y^\bu$ in $\gr\A$.
There exists a natural
transformation ${\eta}\colon {(1)} \to \Id_{\gr\A}$ given by
\[{\eta}_{X^\bu}=\{0, \Id, 0, 0, \cdots\} \colon X^\bu{(1)} \to X^\bu\]
for any object $X^\bu$ in $\gr\A$, where $\Id \colon (X^\bu{(1)})^j
\to X^{j+1}$ $(j\in \Z)$ is the identity map in $\A$.
It is easily seen that ${\eta}_{X^\bu{(1)}}={\eta}_{X^\bu}{(1)}$  for any object $X^\bu$ in
$\gr\A$. From Theorem \ref{Frob cat (1)}, it follows that the  category of complexes
$C(\gr\A)$ has a new Frobenius exact structure $\E_\eta$.

More precisely, a complex in the category $\gr\A$ can be written as a bigraded object
$X^{\bu \bu}$ together with the differential $d^{\bu \bu}$, where $X^{ij}$
is an object of $\A$ and $d^{\bu \bu}$ is a collection of morphisms
$\{d^{ij}_{n}:X^{ij}\to X^{i+1,j+n},n\geq 0, i,j\in \Z\}$ of $\A$ satisfying
\begin{align}\label{d0 d_1 cdots d_n}
d_0^{i+1,j+n}d_n^{ij}+d_1^{i+1,j+n-1}d_{n-1}^{ij}+\cdots+d_n^{i+1,j}d_0^{ij}=0
\end{align}
for any $n\geq 0$ and $i,j\in \Z$. For simplicity of notation, we sometimes write a
 complex $(X^{\bu \bu}, d_X^{\bu \bu})$  instead of $X$. A morphism $f\colon X \to Y$ is a
collection of morphisms $\{ f^{ij}_{n}:X^{ij}\to Y^{i,j+n},n\geq
0\}$ in $\A$ satisfying
\begin{align}\label{f and d}
\begin{split}
 &f_{n}^{i+1,j}d_{X,0}^{ij}+f_{n-1}^{i+1,j+1}d_{X,1}^{ij}+\cdots+ f_{0}^{i+1,j+n}d_{X,n}^{ij}\\
 =&d_{Y,n}^{ij} f_{0}^{ij}+d_{Y,n-1}^{i,j+1}f_{1}^{ij}+\cdots+d_{Y,0}^{i,j+n} f_{n}^{ij}
 \end{split}
 \end{align}
for any $n\geq 0$ and $i,j\in \Z$. The composition $g f$ of two morphisms $f:X\to Y $ and $g: Y\to Z$ is given by
\[( g f)_{n}^{ij}= g_{n}^{ij}f_{0}^{ij}+g_{n-1}^{i,j+1}f_1^{ij}+\cdots+ g_{0}^{i,j+n}f_{n}^{ij}\]
for any $n\geq 0, i, j\in \Z$.

Let $X$ be an object and $f\colon X\to Y$ be
a morphism in $\C(\gr\A)$. We observe that $(X^{\bu, j}, d^{\bu,j}_{X,0})$ is a complex in $\A $  for any $j$,
and $ f_0^{\bu,j}\colon (X^{\bu, j}, d^{\bu,j}_{X,0})\to  (Y^{\bu, j}, d^{\bu,j}_{Y,0})$ is a chain map in $C(\A)$.

By Definition \ref{def of homotopy} and Proposition \ref{prop of
homotopy}, we immediately get the $\eta$-homotopy relation on
$\C(\gr\A)$.

\begin{cor}\label{eta-homotopy in C(Z+A)}
Let $f=\{f_0, f_1, \cdots, f_n,\cdots\} \colon X \to Y$ be a chain map in $\C(\gr\A)$. Then
$\bar f$ vanishes in $\K_\eta(\gr\A)$ if and only if there exist
$s^{i j}_n\colon X^{i j}\to Y^{i-1,j+n-1}$
such that
\begin{align}\label{1-null-homotopic}
\begin{split}
0 & = s_0^{i+1,j}d_{X,0}^{ij}+d_{Y,0}^{i-1,j-1}s_0^{ij}\\
f_{n}^{i j} & =\sum_{p+q=n+1}
(s_p^{i+1,j+q}d_{X,q}^{ij}+d_{Y,p}^{i-1,j+q-1}s_q^{ij})
\end{split}
\end{align}
for any $i,j\in \Z, n\ge 0$.
\end{cor}

Rickard introduced a category $G(\A)$ in \cite[Definition 2.4]{Ric} for any additive category $\A$, which plays an important role in the Rickard's proof
of derived Morita theory.

\begin{Def}[\emph{\cite{Ric}}]
Let $\A$ be an additive category. The category $G(\A)$ is defined as follows.
The objects of $G(\A)$ are
systems $\{X^{\ast \ast}, d_n\colon n=0,1,\cdots,\}$, where $X^{\ast\ast}$ is a bigraded
object of $\A$, and $d_n$ is a graded endomorphism of $X^{\ast\ast}$ of degree $(1-n, n)$,
satisfying
\begin{align*}
d_0d_n+d_1d_{n-1}+\cdots+d_nd_0=0, (n\ge 0).
\end{align*}
The morphisms from $\{X^{\ast\ast},d_{X,n}\}$ to $\{Y^{\ast\ast},d_{Y,n}\}$ are collections
$\al=\{\al_n\colon n=\ge 0\}$ of graded maps $X^{\ast\ast}\to Y^{\ast\ast}$,
such that $\al_n$ has degree $(-n,n)$ and
\[\al_0d_{X,n}+\al_1d_{X,n-1}+\cdots+\al_nd_{X,0}=d_{Y,0}\al_n+d_{Y,1}\al_{n-1}+\cdots+d_{Y,n}\al_0,\]
for each $n$.
The composition of maps $\beta=\{\beta_n\colon n\ge 0\}: X^{\ast \ast}\to Y^{\ast \ast}$ and $\al=\{\al_n\colon n\ge 0\}\colon Y^{\ast\ast}\to Z^{\ast\ast}$
is defined by
\[(\al\beta)_n=\al_0\beta_n+\al_1\beta_{n-1}+\cdots+\al_n\beta_0.\]
\end{Def}

\begin{prop}\label{C(grA) and G(A)}
Let $\A$ be an additive category. Then the categories $G(\A)$ and $C(\gr\A)$ are isomorphic.
\end{prop}
\begin{proof} We define a functor $\Psi\colon G(\A) \to C(\gr\A)$ as follows:
for any object $X=\{X^{\ast \ast}, d_i^{\ast \ast}\colon i=0,1, \cdots\}$ in $G(\A)$,
\begin{align*}
(\Psi(X))^{ij}= & X^{i-j, j}\\
d_{\Psi(X), n}^{ij}=& d_n^{i-j,j}\colon (\Psi(X))^{ij}\to (\Psi(X))^{i+1,j+n};
\end{align*}
and for any morphism $\al=\{\al_n^{\ast \ast}\colon n\ge 0\}\colon X\to Y$,
\[(\Psi(\al))_n^{ij}=\al_n^{i-j,j}\colon (\Psi(X))^{ij} \to (\Psi(Y))^{i,j+n}.\]
By direct calculation, we have that $\Psi$ is an isomorphic functor with the inverse functor $\Psi^{-1}\colon G(\A)\to C(\gr\A)$ given by
\begin{align*}
(\Psi^{-1}(X))^{ij}= & X^{i+j, j}\\
d_{\Psi^{-1}(X), n}^{ij}=& d_n^{i+j,j}\colon (\Psi^{-1}(X))^{ij}\to (\Psi^{-1}(X))^{i+1-n,j+n};
\end{align*}
for any object $(X^{\bu \bu}, d^{\bu \bu}_n, n=0,1,)$ in $C(\gr\A)$, and
\[(\Psi^{-1}(f))_n^{ij}=f_n^{i+j,j}\colon (\Psi^{-1}(X))^{ij}\to (\Psi^{-1}(Y))^{i-n,j+n}.\]
for any morphism $f=\{f_0, f_1, \cdots, f_n, \cdots\}\colon X\to Y$ in $C(\gr\A)$.
\end{proof}

By the above isomorphism and Theorem \ref{1-ex pair}, we know that the category $G(\A)$ admit a Frobenius exact structure and the corresponding stable category is a triangulated category.

\begin{rmk}
Rickard studies a homotopy relation on
$G(\A)$, see \cite[Definition 2.8]{Ric} for more details.
It coincides with the homotopy relation corresponding to the stable category of $(C(\gr\A), \E)$,
rather than the $\eta$-homotopy relation introduced in Definition \ref{def of homotopy}.
\end{rmk}

\subsection{Two functors}

We denote by $\grb\A$ the full additive subcategory of $\gr\A$
consisting of all objects $X^\bu$ with finitely many nonzero $X^i$ in $\A$.
Clearly, $\grb\A$ is closed under the action of the functor $(1)$.
Replacing $\gr\A$ by
$\grb\A$, we get a new Frobenius exact structure $\E_\eta$ on the category $C(\grb\A)$.

We consider the functor ${\Xi}\colon \grb\A \to \A$ defined as
follows. For any object $X^\bu$ and any morphism $f=\{f_0,
f_1,\cdots\} \colon X^\bu \to Y^\bu$ in $\grb\A$,
\begin{align*}
{\Xi}(X^\bu)=& \bigoplus_{j}X^j,\\
{\Xi}(f)= &
\begin{pmatrix}
f_0 \\
f_1 & f_0\\
\vdots & \vdots &\ddots\\
f_n &f_{n-1} & \cdots &f_0 \\
\vdots & \ddots &\ddots  &\ddots & \ddots
\end{pmatrix}.
\end{align*}
Clearly, each column and each row in $\Xi(f)$ have only finitely many nonzero entries for any $f\colon X^\bu \to Y^\bu$ in $\Z_+\A$.
Indeed, if the additive category $\A$ admits infinite direct sums, then one can similarly define the functor $\Xi\colon \gr\A\to \A$.

Naturally, the functor ${\Xi}$ induces a functor from
$\C(\grb\A)$ to $\C(\A)$,
still denoted by $\Xi$.
A relation between the homotopic categories
$K_{\eta}(\grb\A)$ and $K(\A)$ is established by our next
proposition.

\begin{prop}\label{Xi is tri functor} The functor $\Tot$ induces a triangle functor from
$K_{\eta}(\grb\A)$ to $K(\A)$.
\end{prop}
\begin{proof} By \cite[Lemma I.2.8]{Ha1}, it suffices to show that
the functor $\Xi\colon \C(\grb\A) \to \C(\A)$ is exact and preserves
projective objects, where $\C(\grb\A)$ is the  category of complexes in
$\grb\A$ together with the Frobenius exact structure
$\E_{\eta}$, and $\C(\A)$ is the complex category of $\A$
together with the Frobenius exact structure given by chainwise
split short exact sequences of complexes.

Let $X^{\bu \bu} \to Y^{\bu \bu} \to
Z^{\bu \bu}$ be any $\eta$-conflation. Since  it is also a
conflation in $\E$, we have that  $X^{i,\bu} \to Y^{i,\bu} \to
Z^{i,\bu}$ is split in $\grb\A$ for any $i\in \Z$. By definition,
$\Xi(X^{i,\bu}) \to \Xi(Y^{i,\bu}) \to \Xi(Z^{i,\bu})$ is split in
$\A$, and then $\Xi(X^{\bu \bu}) \to \Xi(Y^{\bu \bu}) \to
\Xi(Z^{\bu \bu})$ is a conflation in $\C(\A)$.

For any object $V^{\bu \bu}$ in $C(\grb\A)$,
$\cone(\eta_{V^{\bu \bu}})$ is a projective object in
$(C(\grb\A), \E_{\eta})$. Observe that
$\Xi(V^{i,\bu})=\oplus_j V^{i j}$ and
$\Xi(\eta_{V^{i,\bu}})=\Id_{\Xi(V^{i,\bu})}$. An easy
computation shows that
$\Xi(\cone(\eta_{V^{\bu \bu}}))=\cone(\Id_{\Xi(V^{\bu \bu})})$,
which is a projective object in $\C(\A)$. For any projective object
$P^{\bu \bu}$ in $(\C(\grb\A), \E_{\eta})$, $P^{\bu \bu}$ is a
direct summand of some $\cone(\eta_{V^{\bu \bu}})$ and hence
$\Xi(P^{\bu \bu})$ is a direct summand of
$\cone(\Id_{\Xi(V^{\bu \bu})})$. It follows that $\Xi$ preserves
projective objects.
\end{proof}

Recall that the construction of the functor $\Phi$ in \cite[Proposition 2.11]{Ric}.
Let $\Lambda$ be a ring with identity, and let $\Proj\Lambda$ be the category of
projective modules and $\proj{\Lambda}$ be the category of finitely
generated projective modules over $\Lambda$. For an additive category $\A$, we denote  by
$K(\A)$ (\emph{resp.} $K^b(\A)$ ) the homotopy category of
unbounded (\emph{resp.}  bounded) complexes in $\A$. In \cite{Ric}, Rickard denotes $G(\Proj\La)$ by $G(\La)$.
We also use this notation later.

Recall that a titling complex $T$ means an object of $K^{b}(\proj{\Lambda})$ satisfying
\begin{enumerate}

\item[(i)]  $\Hom_{K^{b}(\proj{\Lambda})}(T,T[n])=0$ for any $n\neq 0$;

\item[(ii)] $\add T$ generates $K^{b}(\proj\Lambda)$ as a triangulated
category,
\end{enumerate}
where $\add T$ is the full subcategory of $K^{b}(\proj\Lambda)$
consisting of all direct summands of finite direct sums of copies of $T$, see \cite[Definition 6.5]{Ric}.

Let $X$ be a complex in $C(\Sum T)$. To be precise, $X=(X^{\bu \bu}, \delta_0^{\bu \bu}, \delta_1^{\bu \bu})$,
where $X^{\bu, j}$ together with $\delta_0^{\bu,j}=(\delta_0^{i j}\colon X^{i j}\to X^{i+1,j})_{i\in \Z}$ is an object in $\Sum T$ for each $j\in \Z$,
and the differential $\delta_1^{\bu,j} \colon X^{\bu,j}\to X^{\bu,j+1}$ satisfies $\delta_1\delta_0=\delta_0\delta_1$
and $\delta_1\delta_1 \colon X^{\bu,j}\to X^{\bu,j+2}$
is null-homotopic since $\Sum T$ is a full subcategory of $K(\Proj\Lambda)$.

We define the operation $\Theta\colon C(\Sum T) \to G(\La)$ as follows:
by Proposition 2.6 and Proposition 2.7 in \cite{Ric} and the isomorphism in Proposition \ref{C(grA) and G(A)}, for any object $X$ in $C(\Sum T)$, there exists an object
$(\Theta(X), d_n\colon n=0,1,\cdots)$ in $C(\gr\Proj\La)$ such that $(\Theta(X))^{ij}=X^{i-j, j}$,
$d_0^{ij}=\delta_0^{i-j,j}$ and $d_1^{ij}=(-1)^{j}\delta_1^{i-j,j}$ for any $i,j\in \Z$, and
for any morphism $\al\colon X \to Y$ in $C(\Sum T)$, there exists a morphism $\Theta(\al)=\{f_0, f_1, \cdots, f_n,\cdots \}\colon \Theta(X)\to \Theta(Y)$
such that $f_0^{ij}=\al^{i-j, j}$ for any $i, j\in \Z$.

By the $\eta$-homotopy relation introduced in Proposition \ref{1-null-homotopic}, we show that
the operation $\Theta$ induces a triangle functor from $K(\Sum T)$ to $K_\eta(\gr\Proj\La)$.
For this, we need the following lemma.

\begin{lem}\label{s0 s1} Let $f=\{f_0, f_1, \cdots, f_n,\cdots \}\colon X \to Y$
be a morphism in $C(\gr\Proj\La)$, where $X$ and $Y$ are in the image of $\Theta$.
Then $f$ is $\eta$-null homotopic if and only if there exist $s_{0}=\{s_0^{ij}\colon X^{ij}\to Y^{i-1,j-1}\}$ and $s_{1}=\{s_1^{ij}\colon X^{ij}\to Y^{i-1,j}\}$ such
that
\begin{align*}
0&=d_{Y,0}^{i-1,j-1}s_{0}^{ij}+s_{0}^{i+1,j}d_{X,0}^{ij},\\
f_{0}&=d_{Y,1}^{i-1,j-1}s_{0}^{ij}+
 d_{Y,0}^{i-1,j}s_{1}^{i j}+s_{0}^{i+1,j+1}d_{X,1}^{ij}+s_{1}^{i+1,j}d_{X,0}^{ij}.
\end{align*}
In particular, if $f_0=0$, then $f\sim_\eta 0$.
\end{lem}

\begin{proof} 
By the equality \eqref{1-null-homotopic} the ``only if" is obvious, and
we use the induction to prove the ``if" by Corollary \ref{eta-homotopy in C(Z+A)}. For convenience, we omit all supscripts.
Suppose that there exist $s_{0}$ and  $s_{1}$ such that
\begin{align*}
0= & d_{Y,0}s_{0}+s_{0}d_{X,0},\\
f_{0}= & d_{Y,1}s_{0}+d_{Y,0}s_{1}+s_{0}d_{X,1}+s_{1}d_{X,0}.
\end{align*}
Proceeding by induction assume that there exists maps
$s_{0},s_{1},\cdots ,s_{k}$ such that
\[f_n=s_{n+1}d_{X,0}+s_nd_{X,1}+\cdots+s_0d_{X,n+1}+d_{Y,n+1}s_0+d_{Y,n}s_1+\cdots+d_{Y,0}s_{n+1}\]
for $n<k$.
By
\[f_kd_{X,0}+f_{k-1}d_{X,1}+\cdots+f_0d_{X,k}=d_{Y,k}f_0+d_{Y, k-1}f_1+\cdots+d_{Y,0}f_k,\]
and the inductive hypothesis, we have
\begin{align*}
& d_{Y,0}(f_{k}-s_{0}d_{X,k+1}-s_{1}d_{X,k}-\cdots-s_{k}d_{X,1}-d_{Y,k+1}s_{0}-d_{Y,k}s_{1}-\cdots-d_{Y,1}s_{k})\\
= & (-d_{Y,1}f_{k-1}-\cdots-d_{Y,k}f_{0}+f_{0}d_{X,k}+\cdots+f_{k}d_{X,0})+s_{0}d_{X,0}d_{X,k+1}\\
&+(d_{Y,1}s_{0}+s_{0}d_{X,1}+s_{1}d_{X,0}-f_{0})d_{X,k}+\cdots\\
&+(d_{Y,1}s_{k-1}+\cdots +d_{Y,k}s_{0}+s_{0}d_{X,k}+\cdots+s_{k}d_{X,0}-f_{k-1})d_{X,1}\\
&+(d_{Y,1}d_{Y,k}+d_{Y,2}d_{Y,k-1}+\cdots +d_{Y,k+1}d_{Y,0})s_{0}+\cdots\\
& +(d_{Y,1}d_{Y,1}+d_{Y,2}d_{Y,0})s_{k-1}+d_{Y,1}d_{Y,0}s_{k}\\
= &d_{Y,1}(-f_{k-1}+s_{0}d_{X,k}+\cdots+s_{k-1}d_{X,1}+d_{Y,k}s_{0}+\cdots+d_{Y,0}s_{k})\\
&+d_{Y,2}(-f_{k-2}+s_{0}d_{X, k-1}+\cdots+s_{k-2}d_{X,1}+d_{Y,k-1}s_{0}+\cdots+d_{Y,0}s_{k-1})
+\cdots\\ &+d_{Y,k}(-f_{0}+d_{Y,1}s_{0}+d_{Y,0}s_{1}
+s_{0}d_{X,1})\\&+s_{0}(d_{X,0}d_{X,k+1}+d_{X,1}d_{X,k}+\cdots +d_{X,k}d_{X,1})
+\cdots+s_{k}d_{X,0}d_{X,1}+f_{k}d_{X,0}\\
= &(f_{k}-s_{0}d_{X,k+1}-s_{1}d_{X,k}-\cdots-s_{k}d_{X,1}-d_{Y,k+1}s_{0}-d_{Y,k}s_{1}-\cdots-d_{Y,1}s_{k})d_{X,0}.
\end{align*}
It follows that
$$f_{k}-s_{0}d_{X,k+1}-s_{1}d_{X,k}-\cdots-s_{k}d_{X,1}-d_{Y,k+1}s_{0}-d_{Y,k}s_{1}-\cdots-d_{Y,1}s_{k}$$
is a morphism in $\Hom_{K(\Proj\Lambda)}(X^{\bu,n+ j},Y^{\bu,j})$. By the definition of tilting complexes, we have \[\Hom_{K(\Proj\Lambda)}(X^{\bu,n+ j},Y^{\bu,j})=0.\]
So there exists $s_{k+1}$
such that
\[f_{k}-(s_{0}d_{X,k+1}+\cdots+s_{k}d_{X,1}+d_{Y,k+1}s_{0}+\cdots+d_{Y,1}s_{k})
=d_{Y,0}s_{k+1}+s_{k+1}d_{X,0},\]
and in consequence,
\[f_{k}=s_{0}d_{X,k+1}+s_{1}d_{X,k}+\cdots+s_{k+1}d_{X,0}+d_{Y,k+1}s_{0}+d_{Y,k}s_{1}+\cdots+d_{Y,0}s_{k+1}.\]
It follows from Corollary \ref{eta-homotopy in C(Z+A)} that $f$ vanishes in $K_\eta(\gr\Proj\La)$.
\end{proof}

\begin{prop} \label{Theta is tri functor}
The operation $\Theta$ induces a triangle functor from
$K(\Sum T)$ to $K_\eta(\gr\Proj\Lambda)$, still denoted by $\Theta$.
\end{prop}
\begin{proof} We first prove that $\Theta\colon K(\Sum T)\to K_\eta(\gr\Proj\La)$ is a functor.
By the construction of $\Theta$, it suffices to show that $\Theta(\al)$ is $\eta$-null homotopy if $\al\colon X\to Y$
is null-homotopy in $C(\Sum T)$,  where $X=(X^{\bu \bu}, \delta_{X,0}^{\bu \bu}, \delta_{X,1}^{\bu \bu})$
and $Y=(Y^{\bu \bu}, \delta_{Y,0}^{\bu \bu}, \delta_{Y,1}^{\bu \bu})$ are complexes in $\Sum T$.

Indeed, if $\al=0$ in $K(\Sum T)$, then there exists a morphism $s_0^{\bu,j}\colon X^{\bu,j}\to Y^{\bu,j-1}$
in $\Sum T$ such that
$\al^{\bu,j}=s_0^{\bu,j+1}\delta_{X,1}^{\bu,j}+\delta_{Y,1}^{\bu,j-1}s_0^{\bu,j}$ in $K(\Sum T)$ for each $j\in \Z$.
It follows that
\[\delta_{Y,0}^{i,j-1}s_0^{i,j-1}+s_0^{i+1,j-1}\delta_{X,0}^{ij}=0,\]
and there exists $s_1^{i j}\colon X^{ij}\to Y^{i-1,j}$ such that
\[\al^{i j}=s_0^{i,j+1}\delta_{X,0}^{i j}+\delta_{Y,0}^{i,j-1}s_0^{i j}
+s_1^{i+1,j}\delta_{X,0}^{i j}+\delta_{Y,0}^{i-1,j}s_1^{i j}\]
for any $i, j\in \Z$. By Lemma \ref{s0 s1}, we have that $\Theta(\al)\colon \Theta(X) \to \Theta(Y)$
vanishes in $K_\eta(\gr\Proj\La)$, and $\Theta\colon K(\Sum T)\to K_\eta(\gr\Proj\La)$ is exactly a functor.

Suppose that $X\xr{\al} Y \xr{\left(0\atop 1\right)} \cone(\al) \xr{(1\ 0)} X[1]$ is
a standard triangle in $K(\Sum T)$.
The mapping cone $\cone(\al)$ is given by $(\cone(\al))^{ij}=X^{i,j+1}\oplus Y^{ij}$,
and
\begin{align*}
\delta_{\cone(\al),0}^{ij}&=\begin{pmatrix} -\delta_{X,0}^{i,j+1} & 0 \\ 0 & \delta_{Y,0}^{ij} \end{pmatrix}\colon X^{i,j+1}\oplus Y^{ij}
\to X^{i+1,j+1}\oplus Y^{i+1,j},\\
\delta_{\cone(\al),1}^{ij}& =\begin{pmatrix} -\delta_{X,1}^{i,j+1} & 0 \\ \al^{i,j+1} & \delta_{Y,1}^{ij} \end{pmatrix}\colon X^{i,j+1}\oplus Y^{ij}
\to X^{i,j+2}\oplus Y^{i,j+1}.
\end{align*}
By definition of $\Theta$, we have
$\Theta(\al)=\{f_0, f_1, \cdots \}\colon \Theta(X) \to \Theta(Y)$ satisfies $f_0^{ij}=\al^{i-j,j}$. Since
$\Theta(\al)\eta_{\Theta(X)}
=\{0, f_0, f_1, \cdots\}\colon \Theta(X)(1)\to \Theta(Y)$,
the mapping cone $\cone(\Theta(\al)\eta_{\Theta(X)})$ is given by
\begin{align*}
(\cone(\Theta(\al)\eta_{\Theta(X)}))^{ij}=& (\Theta(X)(1)[1])^{ij}\oplus (\Theta(Y))^{ij}=X^{i-j,j+1}\oplus Y^{i-j, j},\\
d_{\cone(\Theta(\al)\eta_{\Theta(X)}), 0}^{ij}=&
\begin{pmatrix}
-d_{\Theta(X),0}^{i+1, j+1} & 0\\
0 & d_{\Theta(Y), 0}^{ij}
\end{pmatrix} ,\\
d_{\cone(\Theta(\al)\eta_{\Theta(X)}), n}^{ij}=&
\begin{pmatrix}
-d_{\Theta(X),n}^{i+1, j+1} & 0\\
f_{n-1}^{i+1,j+1} & d_{\Theta(Y), n}^{ij}
\end{pmatrix}, \ (n=1,2,\cdots).
\end{align*}
It immediately follows that $\Theta(\cone(\al))=\cone(\Theta(\al)\eta_{\Theta(X)})$.
On the other hand, we observe that
\[(\Theta(X[1]))^{ij}=(X[1])^{i-j,j}=X^{i-j, j+1}=(\Theta(X))^{i+1,j+1}=(\Theta(X[1](1)))^{ij}\]
and $\Theta(X[1])=\Theta(X)[1](1)$. By direct calculation, we have the following diagram
\[\begin{CD}
\Th(X) @>\Th(\al)>> \Th(Y) @>>> \Th(\cone(\al)) @>>> \Th(X[1])\\
@| @| @| @|\\
\Th(X) @>\Th(\al)>> \Th(Y) @>>> \cone(\Th(\al)\eta_{\Theta(X)}) @>>> \Th(X)(1)[1].
\end{CD}\]
It follows that $\Theta$ sends a triangle in $K(\Sum T)$ to a triangle in $K_\eta(\gr\Proj\La)$ and then $\Theta$ is a triangle functor.
\end{proof}

In \cite[Proposition 2.10]{Ric}, Rickard introduced the functor $\Phi$ from $C(\Sum T)$ to $K(\Proj\La)$
as the composition of a ``total" functor from $G(\La)$ to $K(\Proj\La)$ and an operation $(-)^{\ast\ast}$ from
$C(\Sum T)$ to $G(\La)$. More precisely, for each object $(X, \delta_0, \delta_1)$ of $C(\Sum T)$, one can choose an object $(X^{\ast\ast}, d_n\colon n\ge 0)$ of $G(\La)$ such that $d_0^{ij}=\delta_0^{ij}$ and $d_1^{ij}=(-1)^{i+j}\delta_1^{ij}$ for any $i,j\in \Z$; for each map $\al\colon X\to Y$ in $C(\Sum T)$,
one can choose a morphism $\{\al_n\colon n\ge 0\}\colon X^{\ast\ast} \to Y^{\ast\ast}$ in $G(\La)$ such that $\al_0=\al$.
Defines $\Phi(X)$ as the ``total complex" of $X^{\ast\ast}$ and $\Phi(\al)$ as the map in $K(\Proj\La)$ obtained
by applying the total complex functor to $\{\al_n\colon n\ge 0\}$.

We observe that the operation $(-)^{\ast\ast}$ from $C(\Sum T)$ to $G(\La)$ is not a functor. Even up to the homotopy relation introduced in \cite[Definition 2.8]{Ric}, the operation is still not a functor.
By the isomorphism in Proposition \ref{C(grA) and G(A)}, the above observation means that the operation $\Theta$ can not induce a functor  from $K(\Sum T)$ to $K(\gr\Proj\La)$. The following result show that the functor $\Phi\colon K(\Sum T) \to K(\Proj\La)$
can be realized as the composition of two triangle functors.

\begin{thm}\label{realization}
The functor $\Phi\colon K(\Sum T) \to K(\Proj\La)$ is the composition of
two functors  $\Theta \colon K(\Sum T) \to K_\eta(\gr\Proj\La)$ and $\Xi\colon K_\eta(\gr\Proj\La) \to K(\Proj\La)$.
\end{thm}
\begin{proof} By Definition, the functor $\Theta \colon K(\Sum T) \to K_\eta(\gr\Proj\La)$
is induced by the composition of the operation $(-)^{\ast\ast}\colon C(\Sum T)\to G(\La)$ and the isomorphism  $\Psi^{-1}\colon G(\La) \to C(\gr\Proj\La)$. On the other hand, taking  $\A=\Proj\La$ and we have that the functor $\Xi\colon K_\eta(\gr\Proj\La) \to K(\Proj\La)$ is induced by the composition of the isomorphism $\Psi\colon C(\gr\Proj\La) \to G(\La)$ in the proof of Proposition \ref{C(grA) and G(A)} and the ``total" functor from $G(\La)$ to $C(\Proj\La)$ introduced in \cite[Section 2, Remark]{Ric}.
By the construction of the functor $\Phi\colon K(\Sum T)\to K(\Proj\La)$,
we immediately have that $\Phi=\Xi\circ \Theta$.
\end{proof}

\noindent{\textbf{Acknowledgements}}
The authors would like to thank Professor Xiao-Wu
Chen and Professor Yu Ye for suggesting the problem and for many stimulating conversations.

\vspace{0.5cm}

{\footnotesize
\noindent Yan-Fu Ben \\
School of Mathematical Sciences, Anhui University, Hefei, China, 230039\\
E-mail address: byfee@163.com

\vspace{2mm}
\noindent Yan-Hong Bao \\
School of Mathematical Sciences, University of Sciences and Technology of China,
Hefei,China, 230036\\
School of Mathematical Sciences, Anhui University, Hefei, China, 230039\\
E-mail address: yhbao@ustc.edu.cn

\vspace{2mm}
\noindent Xian-Neng Du\\
School of Mathematical Sciences, Anhui University, Hefei, China, 230039\\
E-mail address: xndu@ahu.edu.cn
}

\end{document}